\documentclass[final,12pt,authoryear]{elsarticle}

\usepackage{amssymb}
\usepackage{amsthm}

\usepackage{amsopn,amsfonts,amsmath}
\usepackage{bm,fixmath}
\usepackage{enumitem}
\usepackage{makecell}

\usepackage{graphicx,wrapfig}
\usepackage{subfig}

\usepackage{algorithm}
\usepackage{algpseudocode}

\usepackage{tikz}
\usetikzlibrary{shapes,arrows.meta}

\usepackage{multirow, booktabs}
\usepackage{siunitx}

\usepackage[pdfencoding=auto]{hyperref}
\usepackage[capitalize,nameinlink]{cleveref}

\theoremstyle{plain}
\newtheorem{theorem}{Theorem}
\newtheorem{lemma}{Lemma}

\newtheorem{proposition}{Proposition}

\theoremstyle{definition}
\newtheorem{definition}{Definition}
\newtheorem{assumption}{Assumption}

\theoremstyle{remark}
\newtheorem{remark}{Remark}
\newtheorem{example}{Example}

\crefformat{equation}{\textup{#2(#1)#3}}
\Crefname{assumption}{Assumption}{Assumptions}
\crefformat{appendix}{#2#1#3}

\DeclareMathOperator*{\argmin}{arg\,min}

\DeclareMathOperator{\prox}{prox}
\DeclareMathOperator{\refl}{refl}
\DeclareMathOperator{\im}{im}

\newcommand{\norm}[1]{\left\lVert#1\right\rVert}

\def\endstatement{\hfill$\square$}

\newcommand{\bv}{\mathbold{b}}
\newcommand{\dv}{\mathbold{d}}
\newcommand{\x}{\mathbold{x}}
\newcommand{\y}{\mathbold{y}}
\newcommand{\w}{\mathbold{w}}
\newcommand{\z}{\mathbold{z}}
\newcommand{\p}{\mathbold{p}}

\newcommand{\zero}{\bm{0}}
\newcommand{\one}{\bm{1}}
\newcommand{\reals}{\mathbb{R}}

\newcommand{\Am}{\mathbold{A}}
\newcommand{\Dm}{\mathbold{D}}
\newcommand{\Bm}{\mathbold{B}}

\newcommand{\Sm}{\mathbold{S}}

\newcommand{\Xm}{\mathbold{X}}
\newcommand{\cv}{\mathbold{c}}
\newcommand{\uv}{\mathbold{u}}

\newcommand{\I}{\mathcal{I}}
\newcommand{\T}{\mathcal{T}}
\newcommand{\X}{\mathcal{X}}
\newcommand{\N}{\mathbb{N}}
\newcommand{\R}{\mathbb{R}}

\newcommand{\Ts}{T_\mathrm{s}}
\newcommand{\Nc}{N_\mathrm{C}}
\newcommand{\Np}{N_\mathrm{P}}
\newcommand{\lip}{\lambda}
\newcommand{\Op}{\mathsf{O}}
\newcommand{\gh}{h} 

\journal{Signal Processing}

\begin{document}

\begin{frontmatter}



\title{Extrapolation-based Prediction-Correction Methods for Time-varying Convex Optimization}

\affiliation[kth]{organization={School of Electrical Engineering and Computer Science, KTH Royal Institute of Technology},
            country={Sweden}}

\affiliation[unipd]{organization={Department of Information Engineering (DEI), University of Padova},
            country={Italy}}

\affiliation[ensta]{organization={UMA, ENSTA Paris, Institut Polytechnique de Paris, 91120 Palaiseau},
            country={France}}

\author[kth]{Nicola Bastianello}
\ead{nicolba@kth.se}

\author[unipd]{Ruggero Carli}
\ead{carlirug@dei.unipd.it}

\author[ensta]{Andrea Simonetto}
\ead{andrea.simonetto@ensta-paris.fr}

\begin{abstract}
In this paper, we focus on the solution of online optimization problems that arise often in signal processing and machine learning, in which we have access to streaming sources of data. We discuss algorithms for online optimization based on the prediction-correction paradigm, both in the primal and dual space. In particular, we leverage the typical regularized least-squares structure appearing in many signal processing problems to propose a novel and tailored prediction strategy, which we call extrapolation-based. By using tools from operator theory, we then analyze the convergence of the proposed methods as applied both to primal and dual problems, deriving an explicit bound for the tracking error, that is, the distance from the time-varying optimal solution. We further discuss the empirical performance of the algorithm when applied to signal processing, machine learning, and robotics problems.
\end{abstract}


\end{frontmatter}


\section{Introduction}\label{sec:intro}

Continuously varying optimization programs have appeared as a natural extension of time-invariant ones when the cost function, the constraints, or both, depend on a time parameter and change continuously in time. This setting captures relevant problems in the data streaming era, see \emph{e.g.}~\cite{dallanese_optimization_2019, simonetto_time-varying_2020} and references therein.

We focus here on linearly constrained regularized least-squares problems of the form
\begin{eqnarray}\label{eq:tv-general-problem-2}
	\mathsf{P}(t): \, \x^*(t), \y^*(t) &\!\!=&\!\!\argmin_{\x \in \reals^n, \y \in \reals^m} \, \underbrace{\frac{1}{2}\|\Dm\x-\dv(t)\|^2 + f_0(\x)}_{=:f(\x;t)} + \gh(\y), \\
	&&\mathrm{s.t.}\  \Am\x + \Bm\y = \cv,
\end{eqnarray}
where, $t\in \reals_+$ is non-negative, continuous, and is used to index time;  $f: \reals^n\times \reals_{+} \to \reals$ is a \emph{smooth strongly convex function} uniformly in time; in addition, $\gh: \reals^m \to \reals \cup \{+\infty\}$ is a closed convex and proper function, matrices $\Dm \in \reals^{q \times n}, \Am \in \reals^{p \times n}$ and $\Bm \in \reals^{p \times m}$, and vector $\cv \in \reals^p$. Finally $\dv(t):\reals_{+} \to \reals^{q}$ is a vector function of time, describing the data. 

Problem $\mathsf{P}(t)$ is typical in signal processing, and depending on the specific values for the matrices and vectors, it could yield a streaming, i.e., time-varying, LASSO, Group-LASSO, the elastic net, as well as various regularized least-squares problems. The structure of Problem $\mathsf{P}(t)$ is so typical that we specifically use it to devise novel algorithms for its resolution. In particular, we use the fact that the Hessian of $f$ is constant in time. 

For handy notation, when $\x=\y$, we introduce the primal problem,
\begin{equation}\label{eq:tv-general-problem-1}
	\mathsf{P}_{\mathsf{p}}(t): \quad \x^*(t) =\argmin_{\x \in \reals^n} \, f(\x; t) + g(\x),
\end{equation} 
with $g: \reals^n \to \reals \cup \{+\infty\}$ a closed convex and proper function and $f$ defined as before.

Solving any of the two problems means determining, at each time $t$, the optimizers $\x^*(t)$ or $(\x^*(t), \y^*(t))$, and therefore, computing the optimizers' trajectory (\emph{i.e.}, the optimizers' evolution in time), up to some arbitrary but fixed accuracy. We notice here that problems $\mathsf{P}(t), \mathsf{P}_{\mathsf{p}}(t)$ are not available a priori, but they are revealed as time evolves: \emph{e.g.}, problem $\mathsf{P}(t')$ will be revealed at $t= t'$ and known for all $t\geq t'$. In this context, we are interested in modeling how the problems $\mathsf{P}(t), \mathsf{P}_{\mathsf{p}}(t)$ evolve in time.

We will look at primal and dual first-order methods. Problem~\eqref{eq:tv-general-problem-1} is the online version of a composite optimization problem (\emph{i.e.}, of the form $f+g$) and we will consider primal first-order methods. Note here that $g$ could be the indicator function of a closed convex set, thereby enabling modeling constrained optimization problems varying with time. Problem~\eqref{eq:tv-general-problem-2} is the online version of the alternating direction method of multipliers (ADMM) setting, and we will consider dual first-order methods. The idea is to present in a unified way a broad class of online optimization algorithms that can tackle instances of problems~\eqref{eq:tv-general-problem-2}-\eqref{eq:tv-general-problem-1}. Further note that the two problems could be transformed into each other, if one so wishes, but we prefer to treat them separately to encompass both primal and dual methods.

The focus is on discrete-time settings as in~\cite{zavala_real-time_2010, dontchev_euler_2013, simonetto_prediction_2017}. In this context, we will use sampling arguments to reinterpret~\cref{eq:tv-general-problem-2,eq:tv-general-problem-1}  as a sequence of time-invariant problems. In particular, focusing here only on~\cref{eq:tv-general-problem-1} for simplicity, upon sampling the objective function $ f(\x; t) + g(\x)$ at time instants $t_{k}$,  $k=0,1,2,\dots$, where the sampling period $\Ts := t_k - t_{k-1}$ can be chosen arbitrarily small, one can solve the sequence of time-invariant problems
\begin{equation}\label{eq:ti-general-problem-1}
	\mathsf{P}_{\mathsf{p}}(t_k): \quad \x^*(t_k) = \argmin_{\x \in \reals^n} \, f(\x; t_k) + g(\x), \qquad k \in \N. 
\end{equation}
By decreasing $\Ts$, an arbitrary accuracy may be achieved when approximating problem~\cref{eq:tv-general-problem-1} with \cref{eq:ti-general-problem-1}. In this context, we will hereafter assume that $\Ts$ is a small constant and $\Ts<1$. However, solving \cref{eq:ti-general-problem-1} for each sampling time $t_k$ may not be computationally affordable in many application domains, even for moderate-size problems. We therefore consider here approximating the discretized optimizers' trajectory $\{\x^*(t_k)\}_{k \in \N}$ by using first-order methods. In particular, we will focus on prediction-correction methods~\cite{zavala_real-time_2010, dontchev_euler_2013, simonetto_prediction_2017, paternain2019prediction}. This methodology arises from non-stationary optimization~\cite{Moreau1977, Polyak1987}, parametric programming~\cite{Robinson1980,Guddat1990,zavala_real-time_2010,dontchev_euler_2013,Hours2014,Kungurtsev2017}, and continuation methods in numerical mathematics~\cite{Allgower1990}.

This paper extends the current state-of-the-art methods, \emph{e.g.},~\cite{simonetto_prediction_2017, Paper4}, by offering the following contributions.

\begin{enumerate}[leftmargin=*]
\item We provide novel prediction-correction methods in both primal space and dual space for online optimization with constraints. In doing so, we show how existing prediction-correction online algorithms can be generalized with the help of operator theoretical tools. In particular, the abstract methodology we discuss includes special cases such as the ones based on (projected) gradient method~\cite{simonetto_prediction_2017}, proximal point, forward-backward splitting, Peaceman-Rachford splitting~\cite{ECC}, as well as the ones based on dual ascent~\cite{Paper4}. Moreover, the proposed algorithms includes new online algorithms based on the method of multipliers, dual forward-backward splitting, and ADMM. With our methodology, we obtain unified results, and a general error bound (\cref{pr:generic-error-bound}), which allows one to plug any prediction strategy they are working with and obtain the corresponding asymptotic error.

\item By leveraging the structure of our signal processing problem, we propose and theoretically characterize a prediction strategy which applies extrapolation on a set of past cost functions collected by the online algorithm\footnote{While extrapolation is a known technique in numerical mathematics~\cite{quarteroni_numerical_2007,Qi2019}, here we fully characterize its theoretical asymptotic error, and we use it in a constrained setting.}. The number of historical costs used can be tuned in order to increase accuracy. Differently from the Taylor expansion-based prediction strategy of \textit{e.g.} \cite{simonetto_class_2016}, the prediction can be computed without needing to compute derivatives of the cost. Under suitable assumptions, we analyze the convergence of the resulting online algorithm, in particular by deriving an upper bound to the asymptotic tracking error (\textit{i.e.} the distance from the optimal trajectory $\{ \x^*(t_k) \}_{k \in \N}$).

\item We further apply the proposed extrapolation prediction strategy to problems with linear constraints such as Problem~\eqref{eq:tv-general-problem-2}, for which we prove convergence within a bounded neighborhood of the optimal trajectories $\{  \x^*(t_k)$, $ \y^*(t_k) \}_{k \in \N}$.

\end{enumerate}

\subsection{Related work}

Time-varying, streaming, and online problems have a long tradition in signal processing and machine learning. The recent surveys~\cite{dallanese_optimization_2019, simonetto_time-varying_2020} cover some key references. From the signal processing literature, we can cite here early algorithms for recursive least-squares and compressive sensing~\cite{Angelosante2010,Cattivelli2008,Vaswani2015,Yang2015}, as well as for dynamic filtering~\cite{Asif2014,Balavoine2015,Charles2016}. These signal processing problems are special cases of problem~\cref{eq:ti-general-problem-1}, and the algorithms proposed in this paper can then be applied to solve them.

More recently, the works~\cite{Jakubiec2013,Ling14admm,simonetto_class_2016,simonetto_prediction_2017} are in line with what we present here, in the sense that they depict time-varying optimization solutions where given a new problem at time $t$, one attempts at finding an approximate optimizer of it. In this sense, past data help warm starting the algorithm at time $t$, but do not influence the new sampled problem. In this paper, we take the same approach as these previous works, but propose a novel warm-starting strategy, and theoretically analyze its performance for the different class of problems~\cref{eq:tv-general-problem-2}.

This line of research is also related to online convex optimization (OCO) \cite{shalev_online_2011,hall2015online,dixit_online_2019}, which was formulated to analyze learning from streaming data. However, differently from our approach, in OCO the set-up is adversarial, in the sense that only information observed up to time $t_k$ can be used to compute the decision to be applied at time $t_{k+1}$\footnote{Please refer to \cref{rem:taylor-prediction} for further discussions.}. Once the decision is applied the learner gains access to the new cost function and incurs a regret; importantly, the cost function may be chosen adversarially to maximize this regret.

Finally, we mention the related approach of streaming optimization discussed in \cite{Hamam2022}. Similarly to the approach in this paper, a new cost function is revealed at each time; with the difference that also a new optimization \textit{variable} is added, and the goal is to solve the overall problem being pieced together over time. In our approach, we focus on a time-varying cost function with a fixed size unknown variable, and assume that the cost function observed at time $t_k$ provides all the information required to compute (in principle) the optimal solution.

\smallskip
\noindent{\bf Organization.} In \cref{sec:background}, we introduce the necessary background. In \cref{sec:PC-framework,sec:convergence}, we present the proposed prediction-correction methodology and the novel extrapolation-based prediction approach, and analyze its performance. \Cref{sec:dual} describes the dual version of the proposed approach. \Cref{sec:simulations} concludes with several numerical examples.


\section{Mathematical Background}\label{sec:background}

\subsection{Notation}
Vectors are written as $\x\in\reals^n$ and matrices as $\Am\in\reals^{p\times n}$. We denote by $\lambda_\mathrm{M}(\Am)$ and $\lambda_\mathrm{m}(\Am)$ the largest and smallest eigenvalues of a square matrix $\Am \in \reals^{n \times n}$. We use $\|\cdot\|$ to denote the Euclidean norm in the vector space, as well as the respective induced norms for matrices. In particular, given a matrix $\Am \in \R^{p \times n}$, we have $\norm{\Am} = \sigma_\mathrm{M}(\Am) = \sqrt{\lambda_\mathrm{M}(\Am^\top \Am)}$, where $\sigma_\mathrm{M}$ denotes the largest singular value. The gradient of a differentiable function $f(\x;t): \reals^{n}\times \reals_{+} \to \reals$ with respect to $\x$ at the point $(\x,t)$ is denoted as $\nabla_{\x} f(\x; t) \in \reals^n$, and $\nabla_{\x\x} f(\x; t) \in \reals^{n\times n}$ denotes the Hessian of $f(\x;t)$ w.r.t. $\x$. The notations $\frac{\partial^{(I)}}{\partial t^{(I)}} \nabla_\x f(\x; t) = \nabla_{t \cdots t \x} f(\x; t)$ denote the $I$-th derivative w.r.t. $t$ of the gradient. We indicate the inner product of vectors belonging to $\reals^n$ as $\langle \mathbold{v}, \mathbold{u} \rangle := \mathbold{v}^\top \mathbold{u}$, for all $\mathbold{v} \in \reals^n, \mathbold{u}\in \reals^n$, where $(\cdot)^\top$ means transpose. We denote by $\circ$ the composition operation. We use $\{ \x^\ell \}_{\ell \in \N}$ to indicate sequences of vectors indexed by non-negative integers, for which we define linear convergence as follows (see \cite{potra_qorder_1989} for details). 

\begin{definition}[Linear convergence]
Let $\{ \x^\ell \}_{\ell \in \N}$ and $\{ \y^\ell \}_{\ell \in \N}$ be sequences in $\R^n$, and consider the points $\x^*, \y^* \in \R^n$. We say that $\{ \x^\ell \}_{\ell \in \N}$ converges \emph{Q-linearly} to $\x^*$ if there exists $\lip \in (0,1)$ such that: $\norm{\x^{\ell+1} - \x^*} \leq \lip \norm{\x^\ell - \x^*}$, $\forall \ell \in \N$.

We say that $\{ \y^\ell \}_{\ell \in \N}$ converges \emph{R-linearly} to $\y^*$ if there exists a Q-linearly convergent sequence $\{ \x^\ell \}_{\ell \in \N}$ and $C > 0$ such that: $\norm{\y^\ell - \y^*} \leq C \norm{\x^\ell - \x^*}$, $\forall \ell \in \N$.
\end{definition}

\subsection{Convex analysis}
A function $f : \reals^n \to \reals$ is \emph{$\mu$-strongly convex}, $\mu > 0$, iff $f(\x) - \frac{\mu}{2} \norm{\x}^2$ is convex. It is said to be \emph{$L$-smooth} iff $\nabla f(\x)$ is $L$-Lipschitz continuous or, equivalently, iff $f(\x) - \frac{L}{2} \norm{\x}^2$ is concave. We denote by $\mathcal{S}_{\mu,L}(\reals^n)$ the class of twice differentiable, $\mu$-strongly convex, and $L$-smooth functions, and $\kappa := L / \mu$ will denote the condition number of such functions.
An extended real line function $f : \reals^n \to \reals \cup \{+\infty\}$ is \emph{closed} if its epigraph $\operatorname{epi}(f) = \{ (\x,a) \in \reals^{n+1} \ | \ x \in \operatorname{dom}(f), \ f(\x) \leq a \}$ is closed. It is \emph{proper} if it does not attain $-\infty$. We denote by $\Gamma_0(\reals^n)$ the class of closed, convex and proper functions $f : \reals^n \to \reals \cup \{+\infty\}$. Notice that functions in $\Gamma_0(\reals^n)$ need not be smooth.
Given a function $f \in \Gamma_0(\reals^n)$, we define its \emph{convex conjugate} as the function $f^* : \reals^n \to \reals \cup \{\infty\}$ such that
$
	f^*(\w) = \sup_{\x \in \reals^n} \left\{ \langle \w, \x \rangle - f(\x) \right\}
$. The convex conjugate of a function $f \in \Gamma_0(\reals^n)$ belongs to $\Gamma_0(\reals^n)$ as well, and if $f \in \mathcal{S}_{\mu,L}(\reals^n)$ then $f^* \in \mathcal{S}_{1/L,1/\mu}(\reals^n)$, \cite[Chapter~12.H]{rockafellar_variational_2009}. 

The \emph{subdifferential} of a convex function $f \in \Gamma_0(\reals^n)$ is defined as the set-valued operator $\partial f : \reals^n \rightrightarrows \reals^n$ such that:
$$
	\x \mapsto \!\left\{\! \z \in \reals^n\ |\ \forall \y \in \reals^n:\ \langle \y-\x, \z \rangle \!+\! f(\x) \leq f(\y) \!\right\},
$$ and we denote by $\tilde{\nabla} f(\x) \in \partial f(\x)$ the subgradients. The subdifferential of a convex function is \emph{monotone}, that is, for any $\x, \y \in \reals^n$: $0 \leq \langle \x - \y, \uv - \mathbold{v} \rangle$ where $\uv \in \partial g(\x), \mathbold{v} \in \partial g(\y)$.

\subsection{Operator theory}
We briefly review some notions and results in operator theory, and we refer to \cite{ryu_primer_2016,bauschke_convex_2017} for a thorough treatment.

\begin{definition}
An operator $\T : \reals^n \to \reals^n$ is:
\begin{itemize}
\item \emph{$\lip$-Lipschitz}, with $\lip>0$, iff $\norm{\T\x - \T\y} \leq \lip \norm{\x - \y}$ for any two $\x, \y \in \reals^n$; it is \emph{non-expansive} iff $\lip \in (0,1]$ and \emph{$\lip$-contractive} iff $\lip \in (0,1)$;
\item \emph{$\beta$-strongly monotone}, with $\beta > 0$, iff $\beta \norm{\x - \y}^2 \leq\langle \x - \y, \T\x - \T\y \rangle$, for any two $\x, \y \in \reals^n$.
\end{itemize}
\end{definition}

By using the Cauchy-Schwarz inequality, a $\beta$-strongly monotone operator can be shown to satisfy:
\begin{equation}\label{eq:strongly-monotone}
	\beta \norm{\x - \y} \leq \norm{\T\x - \T\y}, \quad \forall \x, \y \in \reals^n.
\end{equation}

\begin{definition}
Let $\T : \reals^n \to \reals^n$ be an operator, a point $\x^* \in \reals^n$ is a \emph{fixed point} for $\T$ iff $\x^* = \T \x^*$.
\end{definition}

By the Banach-Picard theorem \cite[Theorem~1.51]{bauschke_convex_2017}, contractive operators have a unique fixed point.

\subsection{Operator theory for convex optimization}\label{subsec:operators-for-optimization}
Operator theory can be employed to solve convex optimization problems; the main idea is to translate a minimization problem into the problem of finding the fixed points of a suitable operator. 

Let $f \in \mathcal{S}_{\mu,L}(\reals^n)$ and $g \in \Gamma_0(\reals^n)$ and consider the optimization problem
\begin{equation}\label{eq:generic-problem}
	\x^* = \argmin_{\x \in \reals^n} \left\{ f(\x) + g(\x) \right\}.
\end{equation}

Let $\T: \reals^p \to \reals^p$ and $\X: \reals^p \to \reals^n$  be two operators. Let $\T$ be $\lip$-contractive and such that its fixed point $\z^*$ yields the solution to~\cref{eq:generic-problem} through the operator $\x^* = \X \z^*$. Let the operator $\X$ be $\chi$-Lipschitz.

We employ then the \emph{Banach-Picard fixed point algorithm}, defined as the update:
\begin{equation}\label{eq:banach-picard}
	\z^{\ell+1} = \T \z^\ell, \quad \ell \in \N.
\end{equation}
By the contractiveness of $\T$, the Q-linear convergence to the fixed point is guaranteed~\cite[Theorem~1.51]{bauschke_convex_2017}:
\begin{equation}
	\norm{\z^{\ell+1} - \z^*} \leq \lip \norm{\z^\ell - \z^*} \leq \lip^{\ell+1} \norm{\z^0 - \z^*},
\end{equation}
as well as R-linear convergence of $\{ \x^\ell \}_{\ell \in \mathbb{N}}$ obtained through $\x^{\ell} = \X \z^{\ell}$ as
\begin{equation}\label{eq:r-linear-convergence}
	\norm{\x^{\ell+1} - \x^*} \leq \chi \lip^{\ell+1} \norm{\z^0 - \z^*}.
\end{equation}

The following lemma further characterizes the convergence in terms of $\x$.

\begin{lemma}\label{lem:contraction-results}
Let $\X$ be $\beta$-strongly monotone. Then, convergence of the sequence $\{ \x^\ell \}_{\ell \in \mathbb{N}}$ with $\x^\ell = \X \z^\ell$ is characterized by the following inequalities:
\begin{equation}\label{eq:iterated-contraction-total}
	\norm{\x^\ell - \x^*} \leq \zeta(\ell) \norm{\x^0 - \x^*}, \qquad \norm{\x^\ell - \x^0} \leq \xi(\ell) \norm{\x^0 - \x^*}
\end{equation}
where
\begin{equation}
	\zeta(\ell) := \left\{\begin{array}{lr} 1, & \textrm{for } \ell = 0, \\ \frac{\chi}{\beta} \lip^\ell, & \textrm{otherwise} \end{array}  \right.
\quad \text{and} \quad 
    \xi(\ell) := \left\{\begin{array}{lr} 0, & \textrm{for } \ell = 0, \\ 1 + \frac{\chi}{\beta}\lip^\ell, & \textrm{otherwise} \end{array}  \right. \, .
\end{equation}
\end{lemma}
\begin{proof}
In the case in which $\ell = 0$, then we have $\norm{\x^\ell - \x^*} = \norm{\x^0 - \x^*}$ and $\norm{\x^\ell - \x^0} = \norm{\x^0 - \x^0} = 0$, which give the first cases in the definitions of $\zeta(\ell)$ and $\xi(\ell)$.
If $\ell > 0$, then by $\beta$-strong monotonicity of $\X$, Eq.~\eqref{eq:strongly-monotone}, we have
\begin{equation}\label{eq:z-to-x}
	\norm{\z^0 - \z^*} \leq \frac{1}{\beta} \norm{\X \z^0 - \X \z^*} = \frac{1}{\beta} \norm{\x^0 - \x^*}.
\end{equation}
Combining~\cref{eq:r-linear-convergence} with~\cref{eq:z-to-x} then yields the second case in the definition of $\zeta(\ell)$. Moreover, using the triangle inequality we have:
$
	\norm{\x^\ell - \x^0} \leq \norm{\x^\ell - \x^*} + \norm{\x^0 - \x^*},
$
and the second case in the definition of $\xi(\ell)$ follows by~\cref{eq:r-linear-convergence} and~\cref{eq:z-to-x}.
\end{proof}

Examples of $\T$ and $\X$ operators for primal problems are reported in Example~\ref{ex:primal-solvers} below, while Example~\ref{ex:pc-admm} discusses the dual solver Alternating direction method of multipliers (ADMM).
We now give a formal definition of operator theoretical solver, which will be needed for our developments. 

\begin{definition}[Operator theoretical solver]\label{def:operator-th-solver}
Let $\T : \reals^p \to \reals^p$ and $\X : \reals^p \to \reals^n$ be two operators, respectively, $\lip$-contractive for $\T$ and $\chi$-Lipschitz and $\beta$-strongly monotone for $\X$, such that the solution $\x^*$ of~\cref{eq:generic-problem} can be computed as $\x^* = \X \z^*$ with $\z^*$ being the fixed point of $\T$.
Suppose that a recursive method, \emph{e.g.} the Banach-Picard in~\cref{eq:banach-picard}, is available to compute the fixed point $\z^*$. Then we call this recursive method an \emph{operator theoretical solver} for problem~\cref{eq:generic-problem}, and call each recursive update of the method a \emph{step} of the solver. We also use the short-hand notation $\Op(\lip, \chi, \beta)$ to indicate such a solver, for which the contraction rates in~\cref{lem:contraction-results} are valid. 
\end{definition}

\begin{example}[Operator theoretical solvers]\label{ex:primal-solvers}
Problem~\cref{eq:generic-problem} can be solved by applying one of the following \emph{splitting algorithms}:
\begin{itemize}[leftmargin=*]
	\item \emph{Forward-backward splitting (FBS)} (or \emph{proximal gradient method}): we choose $\T = \prox_{\rho g} \circ (\I - \rho \nabla_{\x} f)$, which is contractive for $\rho < 2 / L$ and has $\X = \I$; the algorithm is characterized by \cite{taylor_convex_2017}:
	\begin{equation}\label{eq:fbs}
    	\y^{\ell} = \x^\ell - \rho \nabla_{\x} f(\x^\ell),
    	\qquad \x^{\ell+1} = \prox_{\rho g}(\y^\ell),
    	\quad \ell \in \mathbb{N}.
	\end{equation}
	\item \emph{Peaceman-Rachford splitting (PRS)}: we choose $\T = \refl_{\rho g} \circ \refl_{\rho f}$, which is contractive for any $\rho > 0$ and has $\X = \prox_{\rho f}$; the algorithm's updates are \cite{giselsson_linear_2017}:
	\begin{equation}\label{eq:prs}
		\x^\ell = \prox_{\rho f}(\z^\ell),
		\qquad\y^\ell = \prox_{\rho g}(2\x^\ell - \z^\ell),
		\qquad\z^{\ell+1} = \z^\ell + (\y^\ell - \x^\ell)
	\end{equation}
	and, from the fixed point $\z^*$ of $\T$ we compute the solution $\x^*$ through $\X = \prox_{\rho f}$.
\end{itemize}
If problem~\cref{eq:generic-problem} does not have a non-smooth term ($g(\x) = 0$ for all $\x \in \reals^n$), then FBS and PRS reduce to the \textit{gradient descent method} \cite{taylor_convex_2017} and \textit{proximal point algorithm} (PPA) \cite{rockafellar_monotone_1976}, respectively.
\end{example}


\section{Prediction-Correction Algorithms}\label{sec:PC-framework}
We start by describing in this section the proposed prediction-correction methodology, referring to problem~\cref{eq:ti-general-problem-1}:
\begin{equation}\label{eq:tv-problem}
	\x_k^* = \argmin_{\x \in \reals^n} \left\{ f(\x;t_k) + g(\x) \right\}, \quad k \in \N
\end{equation}
where hereafter $\x_k^* := \x^*(t_k)$. Notice that the \textit{size} $n$ of the problem does not change over time, only the cost function $f$. As said, problem~\cref{eq:tv-problem} can model a wide range of both constrained and unconstrained optimization problems, in which a smooth term $f$ is (possibly) summed to a non-smooth term $g$. For example, we may have that $g$ is the indicator function of a constraint set, or a non-smooth function promoting some structural properties (such as an $\ell_1$ norm enforcing sparsity).

\begin{remark}[Explicit \emph{v.} implicit time-dependence]
Notice that in many data-driven applications, the costs would not depend explicitly on time; rather, they would depend on time-varying data, and hence only implicitly on time. Nonetheless, the model we employ is general enough to account also for implicit time-dependence.
\end{remark}

\subsection{Methodology}
Suppose that an operator theoretical solver for problem~\cref{eq:tv-problem} is available. The prediction-correction scheme is characterized by the following two steps:
\begin{itemize}[leftmargin=*]
	\item \emph{Prediction}: at time $t_k$, we approximate the as yet unobserved cost $f(\x;t_{k+1})$ using the past observations; let $\hat{f}_{k+1}(\x)$ be such approximation, then we solve the problem
	\begin{equation}\label{eq:prediction-problem}
		\hat{\x}_{k+1}^* = \argmin_{\x \in \reals^n} \left\{ \hat{f}_{k+1}(\x) + g(\x) \right\}
	\end{equation}
	with initial condition $\x_k$, which yields the prediction $\hat{\x}_{k+1}^*$. In practice, it is possible to compute only an approximation of $\hat{\x}_{k+1}^*$, denoted by $\hat{\x}_{k+1}$, by applying $\Np$ steps of the solver.
	
		\item \emph{Correction}: when, at time $t_{k+1}$, the cost $f_{k+1}(\x) := f(\x;t_{k+1})$ is made available, we can correct the prediction computed at the previous step by solving:
	\begin{equation}\label{eq:correction-problem}
		\x_{k+1}^* = \argmin_{\x \in \reals^n} \left\{ f_{k+1}(\x) + g(\x) \right\}
	\end{equation}
	with initial condition equal to $\hat{\x}_{k+1}$. We will denote by $\x_{k+1}$ the (possibly approximate) correction computed by applying $\Nc$ steps of the solver.
\end{itemize}

\begin{wrapfigure}{r}{7.5cm}
\centering
\footnotesize
	\begin{tikzpicture}[scale=0.65]
	
		\draw (-3,0) -- (-2.5,0);
		\draw[dashed] (-2.5,0) -- (-1,0);
		\draw[->] (-1,0) -- (8,0);
		\draw (-3,-0.15) -- (-3,0.15) node[above] {$t_0$};
		\draw (1.8,-0.15) -- (1.8,0.15) node[above] {$t_{k}$};
		\draw (6,-0.15) -- (6,0.15) node[above] {$t_{k+1}$};
		
		\node[align=center] at (-2,-1) {\textsf{available}\\\textsf{information}};
		\node (tk) at (1.8,-1) {$\x_k$, $\{ f_\ell(\cdot) \}_{\ell \leq k}$};
		\node (tk1) at (6,-1) {$f_{k+1}(\cdot)$};
		
		\node[align=center] (pr) at (1.8,-2.5) {\textsf{prediction}\\$\hat{\x}_{k+1}$};
		\node[align=center] (co) at (6,-4.15) {\textsf{corrected}\\\textsf{prediction}\\$\x_{k+1}$};
		\node (plus) at (6,-2.5) {$+$};
		
		\path[->] (tk) edge (pr)
			  (pr) edge (plus)
			  (tk1) edge (plus)
			  (plus) edge (co);
			
	\end{tikzpicture}

\caption{The prediction-correction scheme.}
\label{fig:prediction-correction-scheme}
\vspace{-15pt}
\end{wrapfigure}

\cref{fig:prediction-correction-scheme} depicts the flow of the prediction-correction scheme, in which information observed up to time $t_k$ is used to compute the prediction $\hat{\x}_{k+1}$. In turn, the prediction serves as a warm-starting condition for the correction problem, characterized by the cost observed at time $t_{k+1}$.

\subsubsection{Solvers}
As described above, the proposed methodology requires that an operator theoretical solver for the prediction and correction steps be available. In particular, there are $\lip$-contractive operators $\hat{\T}_{k+1}, \T_{k+1} : \reals^p \to \reals^p$ with fixed points $\hat{\z}_{k+1}^*, \z_{k+1}^*$, and $\chi$-Lipschitz, $\beta$-strongly monotone operators $\hat{\X}_{k+1}, \X_{k+1} : \R^p \to \R^n$, such that
$$
	\hat{\x}_{k+1}^* = \hat{\X}_{k+1} \hat{\z}_{k+1}^* \quad \text{and} \quad \x_{k+1}^* = \X_{k+1} \z_{k+1}^*.
$$
For simplicity, we assume that the convergence rate of the prediction and correction solvers are the same, and we denote them by $\Op(\lip, \chi, \beta)$. Therefore the contraction functions $\zeta$ and $\xi$ in~\cref{lem:contraction-results} are the same in both cases.

There is a broad range of solvers that can be used within the proposed methodology, depending on the structure of problem~\cref{eq:tv-problem}. For example, if $g \equiv 0$, then \emph{gradient method} and \emph{proximal point algorithms} are suitable solvers, while if $g \not\equiv 0$ then \emph{forward-backward}\footnote{Also called \emph{proximal gradient method}.} and \emph{Peaceman-Rachford} splitting can be used.

\subsection{Prediction methods}
The most straightforward prediction method is the choice $\hat{f}_{k+1}(\x) = f_k(\x)$ which simply employs the last observed cost as a prediction of the next. However, as we will discuss in the following, using a more sophisticated prediction strategy can lead to better performance.

In particular, we look at extrapolation-based prediction. First of all we briefly review a numerical technique for  polynomial interpolation \cite{quarteroni_numerical_2007}, which we then leverage to design a novel prediction strategy.

\subsubsection{Polynomial interpolation}
Let $\varphi : \reals \to \reals$ be a function that we want to interpolate from the pairs $\{ (t_i, \varphi_i) \}_{i=1}^I$ where $\varphi_i := \varphi(t_i)$, and with $t_i \neq t_j$ for any $i \neq j$. The interpolated function is then defined as \cite[Theorem~8.1]{quarteroni_numerical_2007}:
\begin{equation}\label{eq:interpolation}
	\hat{\varphi}(t) = \sum_{i=1}^I \varphi_i \ell_i(t), 
	\qquad \text{with} \qquad \ell_i(t) = \prod_{\substack{1 \leq j \leq I \\ j \neq i}} \frac{t - t_j}{t_i - t_j}.
\end{equation}

The interpolation error can be characterized by \cite[Theorem~8.2]{quarteroni_numerical_2007}:
\begin{equation}\label{eq:interpolation-error}
	\varphi(t) - \hat{\varphi}(t) = \frac{\varphi^{(I)}(v)}{I!} \omega_I(t) \qquad \text{with} \qquad \omega_I(t) = \prod_{i=1}^I (t - t_i)
\end{equation}
and where $v$ is a scalar in the smallest interval that contains $t$ and $\{ t_i \}_{i=1}^I$.

Since in the following we are interested in evaluating the interpolated function at a point $t$ that lies outside the interval $[t_1, t_I]$, we will refer to the resulting function as \emph{extrapolation}. 

\subsubsection{Extrapolation-based prediction}\label{subsec:extrapolation}
Let us now apply the polynomial interpolation technique~\cref{eq:interpolation} to the function $f(\x;t) : \reals^n \times \reals_+ \to \reals^n$ w.r.t. the scalar variable $t \in \reals_+$. In particular, we compute the predicted function $\hat{f}_{k+1}$ from the set of past functions $\{ f_i(\x) \}_{i=k-I+1}^k$. Since the sampling times are multiples of $\Ts$, it is easy to see that the coefficients in~\cref{eq:interpolation} become:
$$
	\ell_i(t_{k+1}) = \prod_{\substack{k-I+1 \leq j \leq k \\ j \neq i}} \frac{t_{k+1} - t_j}{t_i - t_j} = \prod_{\substack{1 \leq h \leq I \\ h \neq i - (k+1)}} \frac{h}{i - (k+1) - h} = (-1)^{i-(k+1)} \binom{I}{i - (k+1)}
$$
and letting $\ell_i := \ell_i(t_{k+1})$ the prediction is thus given by
\begin{equation}\label{eq:extrapolation-prediction}
	\hat{f}_{k+1}(\x) = \sum_{i=k-I+1}^k \ell_i f_{i}(\x), \quad \forall \x \in \reals^n.
\end{equation}
In general, however, the predicted cost $\hat{f}_{k+1}$ may not be strongly convex -- as a matter of fact, it can even fail to be convex. 

However, and crucially, since for our $f(\x;t)$ the Hessian is time-independent, then $\nabla_{\x\x} f(\x;t) = \nabla_{\x\x} f(\x)$ for any $(\x;t) \in \reals^n \times \reals_+$, which implies that 
\begin{equation}
\nabla_{\x\x} \hat{f}_{k+1}(\x) = (\sum_{i = k-I+1}^k \ell_i) \nabla_{\x\x} f(\x) = \nabla_{\x\x} f(\x),
\end{equation}
having used the fact that $\sum_{i = k-I+1}^k \ell_i = 1$. Therefore $\hat{f}_{k+1}$ inherits the same strong convexity and smoothness properties of the original cost. 

This property is inherent to our regularized least-squares structure and typical in signal processing, and very useful for good prediction.  

\begin{remark}[Alternative prediction strategies]\label{rem:taylor-prediction}
We mention here two alternative prediction strategies that have been proposed in the literature. The simpler one, widely used in the context of \emph{online learning} \cite{shalev_online_2011}, is the choice of $\hat{f}_{k+1} = f_k$. This strategy, hereafter called ``one-step-back'' prediction, is particularly suited to adversarial environments, where the future cost $f_{k+1}$ is chosen by the adversary and the best decision we can make is based on $f_k$.
Alternatively, under the assumption that the gradient of $f_k$ is differentiable in time we can choose the \emph{Taylor expansion}-based prediction $\nabla_\x \hat{f}_{k+1}(\x) = \nabla_\x f_k(\x_k) + \Ts \nabla_{t\x} f_k(\x_k) + \nabla_{\x\x} f_k(\x_k) (\x - \x_k)$ \cite{simonetto_class_2016}.
\end{remark}

\smallskip

\begin{remark}[Computational comparison]\label{rem:computational-comparison}
From \cref{rem:taylor-prediction}, the Taylor-based prediction is $\nabla_\x \hat{f}_{k+1}(\x) = \nabla_\x f_k(\x_k) + \Ts \nabla_{t\x} f_k(\x_k) + \nabla_{\x\x} f_k(\x_k) (\x - \x_k)$. This means that to compute it, we need to evaluate the gradient, the Hessian, and the time-derivative of the gradient. On the other hand, the extrapolation-based prediction only requires the computation of gradients from the $I$ past costs that are stored -- this means that we only need access to an \textit{oracle} of the gradients, and building a prediction has a much lower cost.
Finally, we remark that the computationally cheaper approach is the one-step-back prediction $\hat{f}_{k+1} = f_k$, which requires accessing the oracle of only one past cost. Nonetheless, as the theoretical and numerical results will show, the more refined extrapolation-based prediction achieves much smaller tracking error than using $\hat{f}_{k+1} = f_k$, thus justifying its higher computational burden.
\end{remark}


\section{Primal Online Algorithms}\label{sec:convergence}
We are now ready to present our main convergence results. We start by formally stating the required assumptions, and then we provide bounds to the tracking error achieved by the proposed prediction-correction method.

\subsection{Assumptions}\label{subsec:assumptions}

\begin{assumption}\label{as:problem-properties}
\emph{(i)} The cost function $f:\reals^n \times \reals_+ \to \reals$ belongs to $\mathcal{S}_{\mu,L}(\R^n)$ uniformly in $t$. \emph{(ii)} The function $g:\reals^n \to \reals \cup \{+\infty\}$ either belongs to $\Gamma_0(\reals^n)$, or $g(\cdot) \equiv 0$. \emph{(iii)} The solution to~\cref{eq:tv-problem} is finite for any $k \in \N$.
\end{assumption}

\cref{as:problem-properties}\emph{(i)} guarantees that problem~\eqref{eq:tv-problem} is strongly convex and has a unique solution for each time instance. Uniqueness of the solution implies that the solution trajectory is also unique.

\begin{assumption}\label{as:time-derivative-f}
The gradient of function $f$ has bounded time derivative, that is, there exists $C_0 > 0$ such that $\norm{\nabla_{t\x} f(\x;t)} \leq C_0$ for any $\x \in \R^n$, $t \in \R_+$.
\end{assumption}

By imposing \cref{as:time-derivative-f} we ensure that the solution trajectory is Lipschitz in time, as we will see, and therefore prediction-type methods would work well.

\begin{assumption}\label{as:higher-derivatives-f}
The function $f$ has a static Hessian, that is, $\nabla_{\x\x} f(\x;t) = \nabla_{\x\x} f(\x)$ for any $(\x;t) \in \reals^n \times \reals_+$. For the chosen extrapolation order $I \in \N$, $I \geq 2$, there exists $C(I) > 0$ such that:
\begin{equation}\label{eq:bounded-I-order-derivatives}
	\norm{\frac{\partial^{(I)}}{\partial t^{(I)}} \nabla_\x f(\x_{k+1}^*; \tau)} \leq C(I), \ \tau \in [t_{k+1-I}, t_{k+1}].
\end{equation}
\end{assumption}

As mentioned in \cref{subsec:extrapolation} a static Hessian guarantees that the extrapolation-based prediction is strongly convex. The bound on the $I$-th time-derivative of the gradient will instead serve to quantify the quality of the prediction, by comparing $\hat{\x}_{k+1}^*$ and the true optimum $\x_{k+1}^*$.

\subsection{Convergence}
We start by presenting a general bound (meta)-proposition, which can be used to derive the asymptotic error for a large variety of prediction strategies, and it is of independent interest. 

\begin{proposition}[General error bound]\label{pr:generic-error-bound}
Let \cref{as:problem-properties} hold and consider any prediction strategy that uses the same functional class as the original problem~\eqref{eq:tv-problem}. Let $\sigma_k, \tau_k \in [0,+\infty)$ be such that for any $k \in \N$:
\begin{equation}\label{eq:sigma-tau}
	\norm{\x_{k+1}^* - \x_k^*} \leq \sigma_k \quad \text{and} \quad \norm{\hat{\x}_{k+1}^* - \x_{k+1}^*} \leq \tau_k.
\end{equation}
Then the error incurred by a prediction-correction method that uses the solver $\Op(\lip, \chi, \beta)$ is upper bounded by:
\begin{equation}\label{eq:overall-error}
	\norm{\x_{k+1} - \x_{k+1}^*} \leq \zeta(\Nc) \Big(\zeta(\Np) \norm{\x_k - \x_k^*} + \zeta(\Np) \sigma_k + \xi(\Np) \tau_k \Big),
\end{equation}
with functions $\zeta$ and $\xi$ defined in~\cref{lem:contraction-results}.
\end{proposition}
\begin{proof}
See \cref{proof:pr:generic-error-bound}.
\end{proof}

We are now ready to bound $\sigma_k$ and $\tau_k$ for our prediction strategy. First, we present a useful lemma that, employing the assumptions in \cref{subsec:assumptions}, bounds the distance between consecutive points in the optimal trajectory $\{ \x_k^* \}_{k \in \N}$.

\begin{lemma}\label{lem:distance-optima}
Let \cref{as:problem-properties,as:time-derivative-f} hold, then the distance between the optimizers of problems~\eqref{eq:tv-problem} at $t_k$ and $t_{k+1}$ is bounded by:
\begin{equation}
	\norm{\x_{k+1}^* - \x_k^*} \leq C_0 \Ts /\mu.
\end{equation}
\end{lemma}
\begin{proof}
See \cref{proof:lem:distance-optima}.
\end{proof}

The second step is to provide a bound on the distance between the optimizer of the prediction problem and the actual optimizer $\x_{k+1}^*$, i.e., a $\tau_k$ for our prediction strategy.

\begin{lemma}\label{lem:extrapolation-error}
Let \cref{as:problem-properties,as:higher-derivatives-f} hold. Using the extrapolation-based prediction~\cref{eq:extrapolation-prediction} of order $I$ for $f$ yields the following prediction error:
\begin{equation}\label{eq:approximation-error-interp}
	\norm{\hat{\x}_{k+1}^* - \x_{k+1}^*} \leq C(I) \Ts^I / \mu.
\end{equation}
\end{lemma}
\begin{proof}
See \cref{proof:lem:extrapolation-error}.
\end{proof}

With these lemmas in place we can now characterize the convergence when the extrapolation-based prediction~\cref{eq:extrapolation-prediction} is employed.

\begin{theorem}\label{th:extrapolation-convergence}
Consider Problem~\eqref{eq:tv-problem}. Consider the prediction-correction algorithm with the extrapolation-based prediction strategy~\cref{eq:extrapolation-prediction} of order $I \in \N$, $I \geq 2$, for $f$. Let \cref{as:problem-properties,as:time-derivative-f,as:higher-derivatives-f} hold. Consider the operator theoretic solver $\Op(\lip, \beta, \chi)$ to solve both the prediction and correction problems with contraction rates $\zeta$ and $\xi$ given in~\cref{lem:contraction-results}. Choose the prediction and correction horizons $\Np$ and $\Nc$ such that
$$
	\zeta(\Nc) \zeta(\Np) < 1.
$$
Then the trajectory $\{\x_k\}_{k \in \mathbb{N}}$ generated by the prediction-correction algorithm converges Q-linearly with rate $\zeta(\Nc) \zeta(\Np)$ to a neighborhood of the optimal trajectory $\{\x_k^*\}_{k \in \mathbb{N}}$, whose radius is upper bounded as
\begin{equation}\label{eq:asymptotic-error-bound}
	\limsup_{k \to \infty} \norm{\x_k - \x_k^*} = \frac{\zeta(\Nc)}{\mu} \frac{[ \zeta(\Np) C_0 \Ts + C(I) \xi(\Np) \Ts^I]}{1 - \zeta(\Nc) \zeta(\Np)}.
\end{equation}
\end{theorem}
\begin{proof}
See \cref{proof:th:extrapolation-convergence}.
\end{proof}

\begin{remark}[$\Nc$ and $\Np$ choice]
Notice that if the operator $\X$ converting between $\z$ and the primal variable $\x$ is the identity (which is the case \emph{e.g.} for gradient and proximal gradient methods), then $\zeta(\Nc) \zeta(\Np) < 1$ is automatically satisfied whenever at least one of $\Nc$ or $\Np$ is non-zero.
\end{remark}


\section{Dual Online Algorithms}\label{sec:dual}
We now propose a dual version to the prediction-correction methodology, that allows us to solve linearly constrained online problems. We apply the extrapolation-based prediction to this class of problems and study the convergence of the resulting method.

\subsection{Problem formulation}\label{subsec:dual-problem-formulation}
We are interested in solving the following online convex optimization problem with linear constraints, cf. \cref{eq:tv-general-problem-2}:
\begin{equation}\label{eq:continuous-time-admm-problem}
	\x^*(t), \y^*(t) = \argmin_{\x \in \reals^n, \y \in \reals^m} \, f(\x; t) + \gh(\y), \quad \text{s.t.} \ \Am\x + \Bm\y = \cv,
\end{equation}
where $\Am \in \R^{p \times n}$, $\Bm \in \R^{p \times m}$ and $\cv \in \R^{p}$. The following assumption will hold throughout this section, and we will further use \cref{as:time-derivative-f,as:higher-derivatives-f} for $f$.

\begin{assumption}\label{as:admm-problem}
\emph{(i)} The cost function $f$ belongs to $\mathcal{S}_{\mu,L}(\reals^n)$ uniformly in time and satisfies \cref{as:time-derivative-f}. \emph{(ii)} The cost $\gh$ either belongs to $\Gamma_0(\reals^m)$ or $\gh(\cdot) \equiv 0$ with $\Bm = \zero$. \emph{(iii)} The matrix $\Am \in \R^{p \times n}$ is full row rank and the vector $\cv$ can be written as the sum of two vectors $\cv' \in \im(\Am)$ and $\cv'' \in \im(\Bm)$\footnote{This assumption ensures that the problem does indeed have a solution; otherwise, it would not be possible to satisfy the linear constraints.}.
\end{assumption}

The Fenchel dual of~\cref{eq:continuous-time-admm-problem} is
\begin{equation}\label{eq:continuous-time-dual-problem}
	\w^*(t) = \argmin_{\w \in \reals^p} \left\{ d^f(\w;t) + d^{\gh}(\w) \right\}
\end{equation}
where
$
	d^f(\w;t) = f^*(\Am^\top \w; t) - \langle \w, \cv \rangle$ and $ d^{\gh}(\w) = {\gh}^*(\Bm^\top \w).
$
Problem~\cref{eq:continuous-time-dual-problem} conforms to the class of problems that can be solved with the prediction-correction splitting methods of \cref{sec:PC-framework}. Indeed, by \cref{as:admm-problem} we can see that $d^f \in \mathcal{S}_{\bar{\mu},\bar{L}}(\reals^p)$, with $\bar{\mu} := \lambda_\mathrm{m}(\Am\Am^\top) / L$ and $\bar{L} := \lambda_\mathrm{M}(\Am\Am^\top) / \mu$ \cite[Prop.~4]{giselsson_linear_2017}; and $d^{\gh} \in \Gamma_0(\reals^p)$ \cite[Cor.~13.38]{bauschke_convex_2017}. We further know that the gradient of $d^f$ is characterized by \cite{giselsson_metric_2015}:
\begin{equation}\label{eq:dual-gradient}
	\nabla_\w d^f(\w;t) = \Am \bar{\x}(\w,t) - \cv, \quad \bar{\x}(\w,t) := \argmin_\x \left\{ f(\x;t) - \langle \Am^\top \w, \x \rangle \right\}.
\end{equation}

Finally, assuming that there exists $C_0 \geq 0$ such that $\norm{\nabla_{t\x} f(\x;t)} \leq C_0$, then we can prove that also the gradient of the dual cost $d^f$ has bounded rate of change.

\begin{lemma}\label{lem:C0-dual-gradient}
Let \cref{as:admm-problem} hold for the primal problem~\cref{eq:continuous-time-admm-problem}. Then $d^f$ is such that, for any $\x \in \reals^n$ and $t \in \reals_+$:
$$
	\norm{\nabla_{t\w} d^f(\w;t)} \leq \norm{\Am} C_0 / \mu =: \bar{C}_0.
$$
\end{lemma}
\begin{proof}
See \cref{proof:lem:C0-dual-gradient}.
\end{proof}

\begin{remark}[Full rank $\Am$]\label{rem:rank-deficient-A}
The assumption that $\Am$ be full row rank is necessary to guarantee that $d^f \in \mathcal{S}_{\bar{\mu},\bar{L}}(\reals^p)$. However, when problem~\cref{eq:admm-problem} reduces to $\min_\x f(\x)$ s.t. $\Am \x = \cv$, this assumption \emph{can be relaxed}. In this case we are able to prove that the dual function $d^f$ is \emph{strongly convex in the subspace of the image of $\Am$, i.e., $\im(\Am)$}. Therefore, if $\im(\Am)$ is an invariant set for the trajectory generated by the solver, the solver is contractive and the convergence analysis of this paper applies to show linear convergence. The solvers dual ascent and method of multipliers indeed satisfy these conditions, see \cite{Paper4} for more details.
\end{remark}

\subsection{Dual prediction-correction methodology}\label{subsec:dual-PC}
Applying the same approach of \cref{sec:PC-framework}, we are interested in solving \cref{eq:continuous-time-admm-problem} sampled at times $t_k$, $k \in \mathbb{N}$:
\begin{equation}\label{eq:admm-problem}
	\x_k^*, \y_k^* = \argmin_{\x \in \R^n, \y \in \R^m} \left\{ f(\x;t_k) + \gh(\y) \right\} \quad%
	\text{s.t.} \ \Am\x + \Bm\y = \cv
\end{equation}
where $\x_k^* = \x^*(t_k)$, $\y_k^* = \y^*(t_k)$. The sequence of dual problems is then
\begin{equation}\label{eq:dual-problem}
	\w_k^* = \argmin_{\w \in \R^p} \left\{ d^f(\w;t_k) + d^{\gh}(\w) \right\}
\end{equation}
with $k \in \mathbb{N}$, $\w_k^* := \w^*(t_k)$. As mentioned above, \cref{eq:dual-problem} can be solved by the prediction-correction methods described in \cref{sec:PC-framework}. The goal then is to design a suitable prediction strategy.

The idea is to apply the extrapolation-based prediction of \cref{sec:PC-framework} to the primal cost function $f_k$, hence choosing $\hat{f}_{k+1}(\x) = \sum_{i=1}^I \ell_i f_{k + 1 - i}(\x)$. The corresponding dual prediction problem then is
\begin{equation}\label{eq:dual-prediction-problem}
	\hat{\w}_{k+1}^* = \argmin_{\w \in \R^p} \left\{ \hat{d}^f_{k+1}(\w) + \hat{d}^{\gh}(\w) \right\}
\end{equation}
with $\hat{d}^f_{k+1}(\w) = \hat{f}^*_{k+1}(\Am^\top \w) - \langle \w, \cv \rangle$.

\subsection{Convergence analysis}
The following result characterizes the convergence in terms of the primal and dual variables.

\begin{theorem}\label{th:dual-convergence}
Consider the problem~\eqref{eq:admm-problem}. Apply the prediction-correction method defined in \cref{sec:PC-framework} to the dual problem~\cref{eq:dual-problem}, with extrapolation-based prediction applied to $f$. Let $\Op(\lip, \beta, \chi)$ be a suitable dual solver with contraction rates $\bar{\zeta}$ and $\bar{\xi}$ given in~\cref{lem:contraction-results} for $d^f \in \mathcal{S}_{\bar{\mu}, \bar{L}}(\reals^p)$. Let \cref{as:admm-problem} hold.

Choose the prediction and correction horizons such that
$$
	\bar{\zeta}(\Nc) \bar{\zeta}(\Np) < 1.
$$
Then the dual trajectory $\{ \w_k \}_{k \in \mathbb{N}}$ generated by the dual prediction-correction method converges to a neighborhood of the optimal trajectory $\{ \w_k^* \}_{k \in \mathbb{N}}$, whose radius is upper bounded as
\begin{equation*}
	\limsup_{k \to \infty} \norm{\w_k - \w_k^*} = \frac{\bar{\zeta}(\Nc)}{\bar{\mu}} \frac{[ \bar{\zeta}(\Np) \bar{C}_0 \Ts + \bar{\xi}(\Np) (\norm{\Am}/\mu) C(I) \Ts^I]}{1 - \bar{\zeta}(\Nc) \bar{\zeta}(\Np)}.
\end{equation*}
Moreover, the primal trajectories $\{\x_k\}_{k \in \mathbb{N}}$, $\{ \Bm \y_k\}_{k \in \mathbb{N}}$ converge to a neighborhood of the optimal trajectories $\{\x_k^*\}_{k \in \mathbb{N}}$ $\{ \Bm \y_k^*\}_{k \in \mathbb{N}}$, whose radii are upper bounded as
\begin{align*}
	&\limsup_{k \to \infty} \norm{\x_k - \x_k^*} = (\norm{\Am} / \mu) \limsup_{k \to \infty} \norm{\w_k - \w_k^*}, \\
	&\limsup_{k \to \infty} \norm{\Bm (\y_k - \y_k^*)} = \norm{\Bm} (\norm{\Am}^2 / \mu + 1 / \rho) \limsup_{k \to \infty} \norm{\w_k - \w_k^*}.
\end{align*}
\end{theorem}
\begin{proof}
See \cref{proof:th:dual-convergence}.
\end{proof}

We conclude this section showing an example of online algorithm that results when applying the prediction-correction approach in the dual space.

\begin{example}[Prediction-correction ADMM]\label{ex:pc-admm}
The well known alternating direction method of multipliers (ADMM) applied to $\min_{\x, \y} f(\x) + \gh(\y)$, $\text{s.t.}\  \Am \x + \Bm \y = \cv$ is characterized by the updates \cite[eq.~(8)]{bastianello_asynchronous_2021}
\footnote{See also the arXiv version of \cite{bastianello_asynchronous_2021} which reports the derivation of eq.~(8) in Appendix~A \url{https://arxiv.org/abs/1901.09252}.}
\begin{align}
    &\x^\ell \!=\! \argmin_{\x \in \reals^n} \!\left\{\! f(\x) + \frac{\rho}{2} \norm{\Am \x - \z^\ell / \rho - \cv}^2 \right\}, \quad\w^\ell \!=\! \z^\ell - \rho (\Am \x^\ell - \cv) \nonumber \\
    &\y^\ell = \argmin_{\y \in \reals^m} \left\{ \gh(\y) + \frac{\rho}{2} \norm{\Bm\y - (2\w^\ell - \z^\ell) / \rho}^2 \right\}, \quad \uv^\ell \!=\! 2\w^\ell - \z^\ell - \rho \Bm \y^\ell, \nonumber \\
    &\z^{\ell+1} = \z^\ell + 2(\uv^\ell - \w^\ell), \quad \ell \in \mathbb{N}, \label{eq:admm-updates}
\end{align}
where only the primal costs $f$ and $\gh$ play an explicit role, and where $\mathbold{w}^\ell$ is the vector of dual variables of the problem. Therefore, when applying ADMM as a solver for both the prediction and correction problems, we apply~\cref{eq:admm-updates} replacing $f$ with $\hat{f}_{k+1}(\x) = \sum_{i=1}^I \ell_i f_{k + 1 - i}(\x)$ and $f_{k+1}(\x)$, respectively.
\end{example}


\section{Numerical Results}\label{sec:simulations}
In this section, we present extensive numerical results to showcase the performance of the proposed prediction strategy. In particular, We apply our algorithm to three synthetic benchmarks, stemming from time-varying regularized least-squares, time-varying online learning with ADMM, and online robotic tracking, as well as one real-data benchmark stemming from online graph signal processing. 

We compare our prediction strategy to other available ones, namely one-step-back prediction~\cite{shalev_online_2011}, Taylor-based prediction using either the exact computations for the time-derivatives and backward finite difference~\cite{simonetto_class_2016, simonetto_prediction_2017}, the simplified prediction strategy of~\cite{Lin2019}, and two ZeaD prediction formulas~\cite{Qi2019}. Other prediction strategies do exist, but they would typically involve more complex computations or more memory (i.e., longer time horizons). While we do not compare all the methods in all the examples, the interested reader is referred to the \texttt{tvopt} Python package\footnote{Code available here \url{https://github.com/nicola-bastianello/tvopt}.} \cite{bastianello_tvopt_2021} which provides all the tools for more extensive comparisons.

\subsection{Time-varying regularized least-squares}
We consider the composite problem (cf. \cite{simonetto_class_2016}):
\begin{equation}\label{eq:f-g-simulations}
	f(\x;t) = \norm{\x - \bv(t)}^2 / 2 + \epsilon \log(1 + \exp(\langle \pmb{1}_n, \x \rangle)) \quad \text{and} \quad g(\x) = \nu \norm{\x}_1
\end{equation}
with $n = 20$, and where $\bv(t) \in \R^n$ is a signal with sinusoidal components (with angular velocity $\omega = 0.02 \pi$ and randomly generated phases), $\epsilon = 0.75$, $\nu = 0.5$. The function $f$ has $\mu = 1$, $L = 1 + \epsilon N / 4$, $C_0 = \omega$. The prediction and correction problems are solved using the proximal gradient method (a.k.a. forward-backward splitting) \cite{bauschke_convex_2017} with step-size $\alpha = 2 / (L + \mu)$, and $\Np = [5,20,40]$, $\Nc = 5$. We run the simulations for three values of the sampling time $\Ts = [0.002, 0.02, 0.2]$.

We compare the proposed extrapolation-based prediction with order $I = 2$ (\textit{i.e.} $\hat{f}_{k+1} = 2 f_k - f_{k-1}$) and order $I = 3$ (\textit{i.e.} $\hat{f}_{k+1} = 3 f_k - 3 f_{k-1} + f_{k-2}$) against the following methods:
\begin{itemize}
    \item ``One-step-back'': the predictive online gradient characterized by $\x_{k+1} = \x_k - \alpha \nabla f_k(\x_k)$ \cite[p.~132]{shalev_online_2011};
    
    \item ``Correction-only'': the online gradient characterized by\footnote{Note how it differs from the ``one-step-back'' since the gradient of $f_{k+1}$ is used instead of the gradient of $f_k$.} $\x_{k+1} = \x_k - \alpha \nabla f_{k+1}(\x_k)$ \cite[eq.~(2)]{dallanese_optimization_2019};

    \item ``Taylor'' (of order $2$): the prediction-correction method using the Taylor expansion-based prediction (see \cref{rem:taylor-prediction}) \cite{simonetto_class_2016}.
\end{itemize}
In \cref{fig:tracking-error} we report a first comparison in terms of the tracking error evolution for the different methods, with $\Ts = 0.2$ and $\Np = 20$ (for the methods using prediction). As we can see, an extrapolation of the third order outperforms all other methods, while in general the prediction-correction methods outperform the one-step-back and correction-only approaches.
\begin{figure}[!ht]
    \centering
    \includegraphics[scale=0.6]{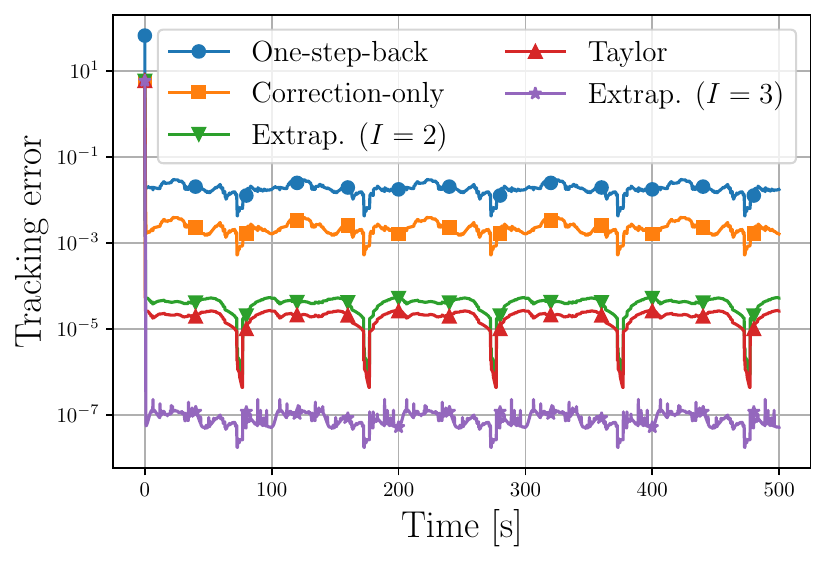}
    \caption{Tracking error comparison with $\Ts = 0.2$ and $\Np = 20$.}
    \label{fig:tracking-error}
\end{figure}
These observations hold up in \cref{tab:asymptotic-error}, which reports the asymptotic tracking error for the different strategies. Combining prediction and correction achieves better results; moreover, the larger the sampling time is, the larger the asymptotic error, which is in accordance with the theory. Also we observe how extrapolation with $I=3$ can boost performance, especially when $\Np$ is large, with respect to Taylor -- this is due to the fact that in \cref{th:extrapolation-convergence} the asymptotic error depends on $\Ts^I$. We also recall, by \cref{rem:computational-comparison}, that the computational complexity of building a Taylor-based prediction exceeds that of the extrapolation-based prediction, since the former needs access to second order derivatives while the latter only to the gradient, so in practice it may not be reasonable to go beyond a Taylor prediction of order $2$.

\begin{table}[!ht]
\centering
\caption{Comparison of asymptotic errors observed in the numerical simulations for~\cref{eq:f-g-simulations}.}
\label{tab:asymptotic-error}

\scalebox{0.8}{
\begin{tabular}{cccc}
	\toprule
	Method & \multicolumn{3}{c}{$\Ts = 0.2$}  \\
	\toprule
	& $\Np=5$ & $\Np=20$ & $\Np=40$  \\
	\toprule

	One-step-back & $3.39 \times 10^{-2}$ & $3.02 \times 10^{-2}$ & $3.02 \times 10^{-2}$ \\
	
	Correction-only & $3.96 \times 10^{-3}$ & $3.96 \times 10^{-3}$ & $3.96 \times 10^{-3}$ \\
	
	Taylor &$4.21 \times 10^{-4}$ & $5.45 \times 10^{-5}$ & $5.45 \times 10^{-5}$ \\
	
	Extrapolation $I\!=\!2$ &$4.19 \times 10^{-4}$ & $2.72 \times 10^{-5}$ &  $2.72 \times 10^{-5}$\\
	
	Extrapolation $I\!=\!3$ &$4.18 \times 10^{-4}$ & $2.35 \times 10^{-7}$ & $5.25 \times 10^{-7}$ \\
\toprule
	 &  \multicolumn{3}{c}{$\Ts = 0.02$}  \\
	\toprule
	& $\Np=5$ & $\Np=20$ & $\Np=40$  \\
	\toprule

	One-step-back &  $3.42 \times 10^{-3}$ & $3.02 \times 10^{-3}$ & $3.02 \times 10^{-3}$ \\
	
	Correction-only  & $4.02 \times 10^{-4}$ & $4.02 \times 10^{-4}$ & $4.02 \times 10^{-4}$ \\
	
	Taylor  & $4.28 \times 10^{-5}$ & $6.67 \times 10^{-7}$ & $6.63 \times 10^{-7}$ \\
	
	Extrapolation $I\!=\!2$ & $4.26 \times 10^{-5}$ & $3.39 \times 10^{-7}$ & $3.31 \times 10^{-7}$ \\
	
	Extrapolation $I\!=\!3$ & $4.24 \times 10^{-5}$ & $6.82 \times 10^{-8}$ & $5.48 \times 10^{-10}$ \\
\toprule
	  & \multicolumn{3}{c}{$\Ts = 0.002$} \\
	\toprule
	& $\Np=5$ & $\Np=20$ & $\Np=40$ \\
	\toprule 

	One-step-back & $3.15 \times 10^{-4}$ & $2.78 \times 10^{-4}$ & $2.78 \times 10^{-4}$\\
	
	Correction-only  & $3.71 \times 10^{-5}$ & $3.71 \times 10^{-5}$& $3.71 \times 10^{-5}$\\
	
	Taylor & $3.91 \times 10^{-6}$ & $9.46 \times 10^{-9}$ & $5.62 \times 10^{-9}$\\
	
	Extrapolation $I\!=\!2$ & $3.91 \times 10^{-6}$ & $7.55 \times 10^{-9}$ & $2.81 \times 10^{-9}$\\
	
	Extrapolation $I\!=\!3$ & $3.91 \times 10^{-6}$ & $6.33 \times 10^{-9}$& $1.67 \times 10^{-12}$\\

	\bottomrule
\end{tabular}}
\end{table}

\subsection{Online graph signal processing}

\begin{figure}[!ht]
    \centering
    \includegraphics[scale=0.6]{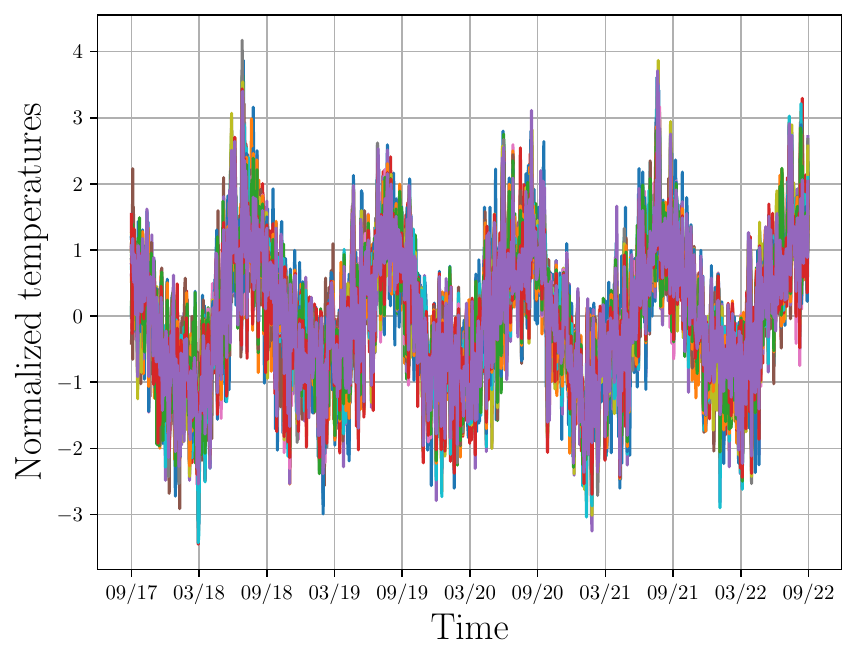}
    \caption{Normalized temperatures over the range Sep.~2017 to Sep.~2022.}
    \label{fig:temperatures}
\end{figure}

As a second example, we consider an online graph signal processing problem, in which the goal is to learn the time-varying topology of a graph from signals observed at the nodes \cite{natali_learning_2022}. Formally, we want to reconstruct the topologies of the graphs in the sequence $\{ \mathcal{G}_k = (\mathcal{V}, \mathcal{E}_k, \Sm_k) \}_{k \in \N}$, where $\mathcal{V} = {1, \ldots, N}$ is the set of nodes, $\mathcal{E}_k$ are the edges at time $k$, and $\Sm_k \in \R^{N \times N}$ is the graph shift operator that represents the topology, which we need to reconstruct. By employing the smoothness-based model of \cite[sec.~IV.C]{natali_learning_2022}, learning the time-varying topology requires that we solve the online problem $\min_{\Sm \in \R^{N \times N}} f_k(\Sm) + g(\Sm)$, where
$$
	f_k(\Sm) = \operatorname{tr}\left(\operatorname{diag}\left( \Sm \one \right) \hat{\bm{\Sigma}}_k \right) - \operatorname{tr}\left( \Sm \hat{\bm{\Sigma}}_k \right) + \frac{\lambda_1}{4} \norm{\Sm}_F^2 - \lambda_2 \one^\top \log(\Sm \one)
$$
with $\hat{\bm{\Sigma}} = \frac{1}{K} \Xm_k \Xm_k^\top$ and $\Xm_k \in \R^{N \times K}$ stacks $K$ samples from the nodes' signals at time $k$. Moreover, $g(\Sm) = \iota_{\mathcal{S}}(\Sm)$ is the indicator function of the set $\mathcal{S}$ of non-negative symmetric matrices with zero diagonal\footnote{In practice, we solve a \emph{vectorized} version of this problem, which corresponds to a composite problem of the form~\cref{eq:tv-general-problem-1}, see \cite[eq.~(32)]{natali_learning_2022} for the details.}.

We use the dataset of hourly temperature measurements at $25$ weather stations across Ireland\footnote{https://www.met.ie/climate/available-data/historical-data}, collected from Sep.~2017 to Sep.~2022. \cref{fig:temperatures} depicts the normalized temperatures observed at the different stations over this range. We follow the set-up of \cite[sec.~VI.B]{natali_learning_2022}, using the proximal gradient method as solver, with $\Np = \Nc = 1$. In \cref{tab:asymptotic-error} we compare the proposed prediction-correction method using an extrapolation-based strategy (with $I = 2$ and $I = 3$) with a correction-only approach and with \cite{natali_learning_2022}, which employs a Taylor expansion-based prediction.

\begin{table}[!ht]
\centering
\caption{Comparison of asymptotic errors observed in the numerical simulations for real weather data.}
\label{tab:as-error-weather}

\footnotesize
\begin{tabular}{cccc}
	\toprule
	Method & Min & Mean $\pm$ Std & Max \\
	\toprule
	
	Correction-only 		& $5.293 \times 10^{-3}$ & $7.393 \times 10^{-2} \pm 4.811 \times 10^{-3}$ & $3.616 \times 10^{-1}$  \\
	
	Taylor \cite{natali_learning_2022} & $3.264 \times 10^{-3}$ & $4.478 \times 10^{-2} \pm 2.945 \times 10^{-2}$ & $2.490 \times 10^{-}$  \\
	
	Extrapolation ($I = 2$) & $3.263 \times 10^{-3}$ & $4.478 \times 10^{-2} \pm 2.945 \times 10^{-2}$ & $2.490 \times 10^{-1}$ \\
	
	Extrapolation ($I = 3$) & $2.812 \times 10^{-3}$ & $4.161 \times 10^{-2} \pm 2.691 \times 10^{-2}$ & $2.156 \times 10^{-1}$ \\
	
	\bottomrule
\end{tabular}
\end{table}

\noindent As we can see in \cref{tab:as-error-weather}, introducing a prediction improves the performance over a correction-only approach. And, while the Taylor expansion-based method has very similar performance to the extrapolation with $I = 2$, the use of an additional past cost, with $I = 3$, in turn improves performance. Notice that the use of higher order extrapolation does not yield the same drastic improvement as in the synthetic problem of the previous section, since it is affected by the noise in the real data and the $C(I)$'s can be rather large for $I$ greater than $2$ or $3$.

\subsection{Online learning with ADMM}
Consider now the online linear regression problem
$
    \min_\x \frac{1}{2} \sum_{i = 1}^N \norm{\Am^i \x - \bv_k^i}^2 + \nu \norm{\x}_1,
$
where each agent $i \in \{ 1, \ldots, N\}$ stores the time-varying data set $(\Am^i, \bv_k^i)$, $\Am^i \in \R^{m_i \times n}$, $\bv_k^i \in \R^{m_i}$. Following a cloud-based learning approach, the goal is to solve this problem by relying on a central coordinator that receives and aggregates the results of local computations, without accessing the local data.
Specifically, we reformulate the problem as (cf. \cite[section~8.2]{boyd_distributed_2010})
\begin{align*}
    &\min_{\{ \x^i \}_{i = 1}^N, \y} \frac{1}{2} \sum_{i = 1}^N \norm{\Am^i \x^i - \bv_k^i}^2 + \nu \norm{\y}_1 \\
    &\text{s.t.} \ \x^i = \y, \ i = 1, \ldots, N
\end{align*}
where the $N$ agents are tasked with processing the local data $(\Am^i, \bv_k^i)$ in order to update $\x^i$, and the central coordinator has the role of averaging $\x^i$ and enforcing sparsity with the $\ell_1$-norm. This reformulation of the problem conforms to \cref{eq:admm-problem} and hence we can apply the prediction-correction ADMM discussed in \cref{ex:pc-admm}.

The numerical results described below were derived as follows. The local matrices $\Am^i$ were randomly generated so that $f_k^i(\x^i) := (1/2) \norm{\Am^i \x^i - \bv_k^i}^2 \in \mathcal{S}_{\mu,L}(\R^n)$, and $\bv_k^i = \Am^i \bar{\x}_k + \mathbold{e}_k^i$ where one third of $\bar{\x}_k$'s components are zero and the remaining change in a sinusoidal way, $\mathbold{e}_k^i$ is random normal noise with either medium variance $0.2$, or low variance $0.002$.
We compared the performance of the one-step-back ADMM, the correction-only ADMM, and the prediction-correction ADMM. For the latter we use extrapolation of order $I=2,3$, as well as Taylor predictions based on backward finite-difference~\cite{simonetto_class_2016}. In particular, in this problem setting, extrapolation reads:
\begin{eqnarray}
I=2 &:& \nabla \hat{f}_{k+1}^i(\x) =\Am^i \x^i + 2 \bv_k^i - \bv_{k-1}^i,   \\ 
I=3 &:& \nabla \hat{f}_{k+1}^i(\x) = \Am^i \x^i + 3 \bv_k^i - 3\bv_{k-1}^i + \bv^i_{k-2}.
\end{eqnarray}
For Taylor with $O(\Ts^2)$ and $O(\Ts^3)$ backward finite-difference,
\begin{eqnarray}
I=2 &:& \nabla \hat{f}_{k+1}^i(\x) = \Am^i \x^i + 2 \bv_k^i - \bv_{k-1}^i,   \\ 
I=3 &:& \nabla \hat{f}_{k+1}^i(\x) = \Am^i \x^i + 3 \bv_k^i - 3\bv_{k-1}^i+ \bv^i_{k-2}.
\end{eqnarray}
Note that Taylor with backward finite-difference is the same here as extrapolation, but this is not true in general. Finally, we report ZeaD~\cite{Qi2019} prediction results with $\zeta = 1$ and $\zeta = 2$:
\begin{eqnarray}
\textrm{ZeaD}, I=3 &:& \nabla \hat{f}_{k+1}(\x) = \Am^i \x^i + \frac{2\zeta +3}{2} \bv_k^i - 2\zeta\, \bv_{k-1}^i+ \frac{2\zeta-1}{2}\bv^i_{k-2}.
\end{eqnarray}

\cref{tab:admm} reports the asymptotic error of the compared approaches for different numbers of agents, each endowed with an equal number of data points from a total of $m = 250$ ($m_i = m/N$), with the setting $P=10, C=2$. Similarly to the results of the previous sections, we observe that prediction-correction is in general better than prediction or correction alone. Different prediction strategies work better in different noise and number of agent settings. In this example, extrapolations of order $2$ and $3$ behave in par with the others, and sometimes marginally better. 
Additionally, we notice that a larger number of agents taking part in the solution of the problem can lead to small improvements in the asymptotic error. This is partly explained by observing that the costs $f_k^i$ have a lower value of $C_0$ (the bound on the gradient's variation over time) than the cost defined on the whole data set, which leads to a lower asymptotic bound according to \cref{th:dual-convergence}.

\begin{table}[!ht]
\centering
\caption{Comparison of asymptotic errors with respect to the true signal, for the online linear regression problem solved using  ADMM, for different numbers of agents $N$. $^*$ this is equivalent to Taylor of the same order with backward finite-difference.}
\label{tab:admm}
\footnotesize
\begin{tabular}{cccccccc}
	\toprule
    & \multicolumn{3}{c}{$\sigma = 0.2$} & &  \multicolumn{3}{c}{$\sigma = 0.002$}\\ \cmidrule(lr){2-4}\cmidrule(lr){6-8}
	Method & $N = 1$ & 5 & 10 & & $N = 1$ & 5 & 10\\
	\toprule

    One-step-back 		& $1.37$ & $0.50$ & $0.28$ && $1.06$ & $0.38$ & $0.36$ \\
	
	Correction-only 		& ${\bf 1.29}$ & $0.63$ & $0.20$ && $1.07$ & $0.50$ & $0.45$ \\
	
	Extrapolation$^*$ ($I = 2$) & $1.37$ & ${\bf 0.34}$ & ${\bf 0.12}$ && ${\bf0.93}$ & $0.16$ & $0.13$ \\

    Extrapolation$^*$ ($I = 3$) & $1.37$ & $0.42$ & $0.15$ && ${\bf0.93}$ & ${\bf 0.15}$ & ${\bf 0.12}$ \\

    ZeaD ($\zeta = 1, I = 3$) & $1.37$ & $0.38$ & $0.15$ && ${\bf0.93}$ & $0.16$ & $0.13$ \\

    ZeaD ($\zeta = 2, I = 3$) & $1.35$ & $0.48$ & $0.18$ && ${\bf 0.93}$ & ${\bf0.15}$ & ${\bf 0.12}$ \\
	
	\bottomrule
\end{tabular}
\end{table}

\subsection{Online robotics}
As a fourth example, we rework here the robotic setting considered in \cite{ECC, dixit_online_2019} In particular, we consider a number $N=10$ of mobile robots that follow a leader robot while it moves in a $2D$ space. The problem can be formulated as,
\begin{eqnarray}
\min_{\x\in\mathbb{R}^{2(N+1)}} && f(\x; t_k) :=\sum_{i=1}^N \frac{1}{2} \left(z_{k}^i - \mathbold{v}_i^\top \x^0\right)^2 + \frac{\lambda}{2}\|\x-\x_k\|^2,  \\
\textrm{subject to} && \Am \x = \bv
\end{eqnarray}
where  $\x^i \in \mathbb{R}^2$ is the position of robot $i=0, \ldots, N$, where $0$ represents the leader. The above problem amounts at estimating the position of the leader robot $\x^0$ based on local measurements $z^i_k$ and a linear model $\mathbold{v}_i^\top$, with a suitable $\ell_2$ regularization. In addition, the followers move as to maintain a rigid formation as imposed by the constraint $\Am \x = \bv$. All the details are given in~\cite{ECC}. 

We solve the above problem with a proximal gradient, by projecting over the constraint, employing several prediction and correction methods. In Figure~\ref{fig:robots}, we report the asymptotical tracking error varying the sampling time for a correction-only method, a simplified prediction~\cite{Lin2019}, two extrapolation-based predictions of $2$nd and $3$rd order, respectively, as well as a ZeaD prediction of third order with ($\zeta=1$). In all cases, the number of proximal gradients are $P=20$ for prediction and $C=5$ for correction. 

\begin{figure}[!ht]
    \centering
    \includegraphics[scale=0.6]{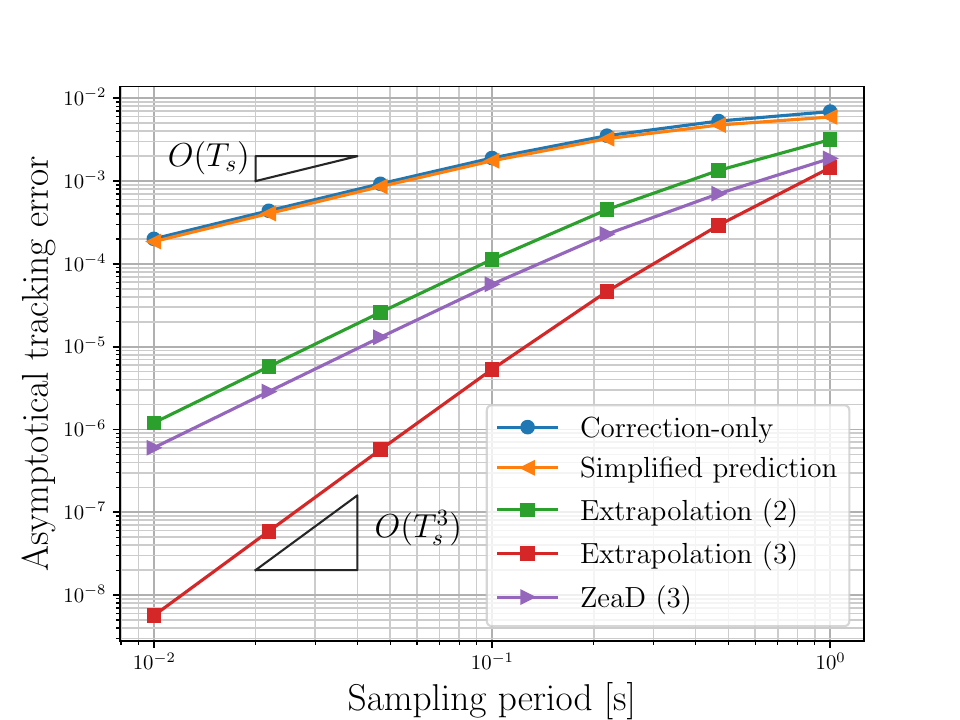}
    \caption{Comparison of several methods to solve an online robotics problem in terms of the asymptotical tracking error vs. the sampling time.}
    \label{fig:robots}
\end{figure}

As one can appreciate, the extrapolation methods achieve the theoretical order of $O(\Ts^2)$ and $O(\Ts^3)$ and do very well with respect to other prediction methods (simplified of second order, and ZeaD of third order), further advocating for this prediction modality.

\appendix


\section{Proofs of \texorpdfstring{\Cref{sec:convergence}}{Section 4}}\label{app:proofs-prediction}

\subsection{Proof of \cref{pr:generic-error-bound}}\label{proof:pr:generic-error-bound}
Consider a prediction-correction strategy where we apply $\Np$ and $\Nc$ steps during prediction and correction, respectively. By \cref{lem:contraction-results}, the following holds:
\begin{subequations}
\begin{align}
	\norm{\hat{\x}_{k+1} - \hat{\x}_{k+1}^*} &\leq \zeta(\Np) \norm{\x_k - \hat{\x}_{k+1}^*} \label{eq:early-termination-P} \\
	\norm{\x_{k+1} - \x_{k+1}^*} &\leq \zeta(\Nc) \norm{\hat{\x}_{k+1} - \x_{k+1}^*} \label{eq:early-termination-C}.
\end{align}
\end{subequations}
The goal now is to bound the prediction error $\norm{\hat{\x}_{k+1} - \x_{k+1}^*}$. If $\Np = 0$ then no prediction steps are applied, and thus, using the triangle inequality, we can write:
\begin{align*}
	\norm{\hat{\x}_{k+1} - \x_{k+1}^*} = \norm{\x_k - \x_{k+1}^*} &\leq \norm{\x_k - \x_k^*} + \norm{\x_{k+1}^* - \x_k^*} \leq \norm{\x_k - \x_k^*} + \sigma_k \\
	&= \zeta(\Np) \norm{\x_k - \x_k^*} + \zeta(\Np)\sigma_k + \xi(\Np) \tau_k
\end{align*}
where we used the facts that $\zeta(\Np) = 1$ and $\xi(\Np) = 0$ if $\Np = 0$ to derive the last equality (cf. \cref{lem:contraction-results}).
Consider now the case of $\Np > 0$. By the triangle inequality and the early termination inequality~\cref{eq:early-termination-P}, the following chain of inequalities holds:
\begin{align*}
	\norm{\hat{\x}_{k+1} - \x_{k+1}^*} &\leq \norm{\hat{\x}_{k+1} - \hat{\x}_{k+1}^*} + \norm{\hat{\x}_{k+1}^* - \x_{k+1}^*} %
	\leq \zeta(\Np) \norm{\x_k - \hat{\x}_{k+1}^*} + \tau_k \\
	&\leq \zeta(\Np) \Big( \norm{\x_k - \x_k^*} + \norm{\x_k^* - \x_{k+1}^*} + \norm{\x_{k+1}^* - \hat{\x}_{k+1}^*} \Big) + \tau_k \\
	&\leq \zeta(\Np) \norm{\x_k - \x_k^*} + \zeta(\Np) \sigma_k + (1+\zeta(\Np)) \tau_k \\
	&= \zeta(\Np) \norm{\x_k - \x_k^*} + \zeta(\Np) \sigma_k + \xi(\Np) \tau_k
\end{align*}
where the last equality follows by the fact that $\xi(\Np) = 1 + \zeta(\Np)$ if $\Np > 0$ (cf. \cref{lem:contraction-results}).
Therefore for any $\Np \geq 0$ we can bound the prediction error as
\begin{equation}\label{eq:prediction-error}
	\norm{\hat{\x}_{k+1} - \x_{k+1}^*} \leq \zeta(\Np) \norm{\x_k - \x_k^*} + \zeta(\Np) \sigma_k + \xi(\Np) \tau_k
\end{equation}
and combining~\cref{eq:prediction-error} with~\cref{eq:early-termination-C} for the correction step yields
\begin{equation}\label{eq:general-tracking-error}
	\norm{\x_{k+1} - \x_{k+1}^*} \leq \zeta(\Nc) \zeta(\Np) \norm{\x_k - \x_k^*} + \zeta(\Nc) \left( \zeta(\Np) \sigma_k + \xi(\Np) \tau_k \right),
\end{equation}
from which the thesis. \endstatement

\subsection{A supporting result}

\begin{theorem}\label{th:solution-mapping}
Let $f \in \mathcal{S}_{\mu,L}(\R^n)$ and $g \in \Gamma_0(\R^n)$, then the solution mapping
$
	S(\p) = \left\{ \y \ | \ \nabla_\x f(\y) + \partial g(\y) \ni \p \right\}
$
of the parameterized generalized equation $\nabla_\x f(\y)$ $+ \partial g(\y) \ni \p$ is single-valued and $\mu^{-1}$-Lipschitz continuous.
\end{theorem}

\begin{proof}
The proof follows from~\cite[Theorem~1]{nesterov_smooth_2005}.
\end{proof}

\subsection{Proof of \texorpdfstring{\cref{lem:distance-optima}}{Lemma~4.1}}\label{proof:lem:distance-optima}
The following proof is an extension of~\cite[Theorem~2F.10]{doncev_implicit_2014} when $\nabla_{t\x}f$ exists everywhere.
First, we define the auxiliary functions:
$
	\Psi(\y) = \nabla_\x f_{k+1}(\y) + \partial g(\y)$ and $\psi(\y) = \nabla_\x f_k(\y) - \nabla_\x f_{k+1}(\y)
$
and, by the fact that $\nabla_\x f_k(\x_k^*) + \partial g(\x_k^*) \ni \zero$ and $\nabla_\x f_{k+1}(\x_{k+1}^*) + \partial g(\x_{k+1}^*) \ni \zero$, we have $(\psi+\Psi)(\x_{k+1}^*) \ni \psi(\x_{k+1}^*)$.
We define now the function $F(\y) = (\psi+\Psi)(\y)$ and consider the parametric generalized equation $F(\y) + \p \ni \zero$. Under \cref{as:problem-properties,as:time-derivative-f}, \cref{th:solution-mapping} implies that the solution mapping $\p \mapsto \y(\p)$ for this generalized equation is everywhere single valued and Lipschitz continuous with constant $\mu^{-1}$, \emph{i.e.} $\norm{\y(\p) - \y(\p')} \leq \norm{\p - \p'}/\mu.$
Therefore, setting $\p = \pmb{0}$ and $\p' = -\psi(\x_{k+1}^*)$, implies 
$$
	\norm{\x_{k+1}^* - \x_k^*} \leq  \norm{\psi(\x_{k+1}^*)} / \mu \leq C_0 \Ts / \mu.
$$
where we used the fact that:
$
	\norm{\psi(\x_{k+1}^*)} = \norm{\nabla_\x f_k(\x_{k+1}^*) - \nabla_\x f_{k+1}(\x_{k+1}^*)} \leq C_0 \Ts
$, see~\cite[eq.~(59)]{simonetto_prediction_2017}. \endstatement

\subsection{Proof of \texorpdfstring{\cref{lem:extrapolation-error}}{Lemma~4.4}}\label{proof:lem:extrapolation-error}
Define the functions
$
	\Psi(\y) = \nabla_\x f_{k+1}(\y)  + \partial g(\y)$ and $\psi(\y) = \nabla_\x \hat{f}_{k+1}(\y) - \nabla_\x f_{k+1}(\y);
$
by the optimality conditions of the correction and prediction problems we have that $(\Psi + \psi)(\x_{k+1}^*) \ni \psi(\x_{k+1}^*)$ and $(\Psi + \psi)(\hat{\x}_{k+1}^*) \ni \pmb{0}$. Then, applying \cref{th:solution-mapping} to the parametrized generalized equation $(\Psi + \psi)(\y) \ni \p$ we have the following bound $\norm{\hat{\x}_{k+1}^* - \x_{k+1}^*} \leq \norm{\psi(\x_{k+1}^*)} / \mu$. By the interpolation error formula~\cref{eq:interpolation-error}, we have the bound:
\begin{equation}\label{eq:extrapolation-bound}
\begin{split}
	\norm{\psi(\x_{k+1}^*)} &= \norm{\nabla_\x \hat{f}_{k+1}(\x_{k+1}^*) - \nabla_\x f_{k+1}(\x_{k+1}^*)} \\ &\leq \norm{\frac{1}{I!} \frac{\partial^{(I)}}{\partial t^{(I)}} \nabla_{\x} f(\x_{k+1}^*;\tau) \omega_I(t_{k+1})} \leq C(I) \Ts^I
\end{split}
\end{equation}
where $\tau \in [t_{k-1}, t_{k+1}]$, and we used the facts that $\omega_I(t_{k+1}) = \prod_{i=1}^I (t_{k+1} - t_{k+1-i}) = I! \Ts^I$ (cf.~\cref{eq:interpolation-error}) and \cref{eq:bounded-I-order-derivatives} to derive the last inequality. \endstatement

\subsection{Proof of \texorpdfstring{\cref{th:extrapolation-convergence}}{Theorem~5.9}}\label{proof:th:extrapolation-convergence}

By \cref{lem:distance-optima,lem:extrapolation-error} we know that there exist $\sigma, \tau \in [0,+\infty)$ such that $\norm{\x_{k+1}^* - \x_k^*} \leq \sigma$ and $\norm{\hat{\x}_{k+1}^* - \x_{k+1}^*} \leq \tau$. As such, \cref{pr:generic-error-bound} holds. We can then use Equation~\eqref{eq:general-tracking-error} with our bounds for $\sigma, \tau$.  

If then, we choose $\Np$ and $\Nc$ such that $\zeta(\Np) \zeta(\Nc) < 1$, then the error converges and using the geometric series the thesis of \cref{th:extrapolation-convergence} follows. \endstatement

\section{Proofs of \texorpdfstring{\Cref{sec:dual}}{Section~5}}\label{app:proofs-dual}

\subsection{Proof of \texorpdfstring{\cref{lem:C0-dual-gradient}}{Lemma~4}}\label{proof:lem:C0-dual-gradient}
Notice that $\bar{\x}(\w,t)$ is the unique solution to the equation
$
	\psi(\x;\w,t) := \nabla_{\x} f(\x;t) - \Am^\top \w = \zero,
$
where $\psi(\x;\w,t)$ is differentiable in $\x$ with $\nabla_{\x} \psi(\x;\w,t) = \nabla_{\x\x} f(\x;t)$ non-singular. Now set $\bar{\x} = \bar{\x}(\w,t)$ to simplify the notation. Fixing $\w$ and applying \cite[Theorem~1.B.1]{doncev_implicit_2014} w.r.t. $t$ gives
$
	{\partial \bar{\x}}/{\partial t} \!=\! - \left[\nabla_\x \psi(\bar{\x};\w,t)\right]^{-1} \nabla_t \psi(\bar{\x};\w,t) = - \nabla_{\x\x} f(\bar{\x};t)^{-1} \nabla_{t\x} f(\bar{\x};t),
$
and as a consequence, we have
\begin{equation}\label{eq:dual-gradient-time-derivative}
	\nabla_{t\w} d^f(\w;t) = - \Am \nabla_{\x\x} f(\bar{\x}(\w,t);t)^{-1} \nabla_{t\x} f(\bar{\x}(\w,t);t).
\end{equation}

Using~\cref{eq:dual-gradient-time-derivative} and the sub-multiplicativity of the norm we have
\begin{align*}
	\norm{\nabla_{t\w} d^f(\w;t)} &= \norm{\Am \nabla_{\x\x} f(\bar{\x}(\w,t);t)^{-1} \nabla_{t\x} f(\bar{\x}(\w,t);t)} \\
	&\leq \norm{\Am} \norm{\nabla_{\x\x} f(\bar{\x}(\w,t);t)^{-1}} \norm{\nabla_{t\x} f(\bar{\x}(\w,t);t)} \leq \norm{\Am} C_0 / \mu
\end{align*}
where the last inequality holds by \cref{as:admm-problem}~(i). \endstatement

\subsection{Proof of \texorpdfstring{\cref{th:dual-convergence}}{Theorem~2}}\label{proof:th:dual-convergence}
As observed in \cref{subsec:dual-problem-formulation}, under \cref{as:admm-problem} the dual cost $d^f_k$ is $\bar{\mu}$-strongly convex and $\bar{L}$-smooth, and $d^{\gh}_k \in \Gamma_0(\reals^p)$. Therefore, we can follow the same derivation in \cref{proof:th:extrapolation-convergence} to show that~\cref{eq:general-tracking-error} holds for the dual problem, with
\begin{equation}\label{eq:dual-general-tracking-error}
	\norm{\w_{k+1} - \w_{k+1}^*} \leq \zeta(\Nc) \zeta(\Np) \norm{\w_k - \w_k^*} + \zeta(\Nc) \left( \zeta(\Np) \bar{\sigma} + \xi(\Np) \bar{\tau} \right).
\end{equation}

The goal now is to provide a bound to both $\bar{\sigma}$ and $\bar{\tau}$. First, since \cref{lem:C0-dual-gradient} holds, we can apply \cref{lem:distance-optima} to prove that
$$
	\norm{\w_{k+1}^* - \w_k^*} \leq \bar{C}_0 \Ts /\bar{\mu} =: \bar{\sigma}.
$$
To bound $\bar{\tau}$, following the derivation in \cref{proof:lem:distance-optima} we can see that
$$
	\norm{\hat{\w}_{k+1}^* - \w_{k+1}^*} \leq \norm{\psi(\w_{k+1}^*)} / \bar{\mu}
$$
where $\psi(\w) = \nabla_\w \hat{d}^f_{k+1}(\w) - \nabla_\w d^f_{k+1}(\w)$. Using~\cref{eq:dual-gradient} we further know that $\nabla_\w d^f_{k+1}(\w_{k+1}^*) = \Am \bar{\x} - \cv$ and $\nabla_\w \hat{d}^f_{k+1}(\w_{k+1}^*) = \Am \bar{\bar{\x}} - \cv$, with $\bar{\x} = \argmin_\x F_{k+1}(\x)$ and $\bar{\bar{\x}} = \argmin_\x \hat{F}_{k+1}(\x)$, having defined
\begin{align*}
	F_{k+1}(\x) &= f_{k+1}(\x) - \langle \Am^\top \w_{k+1}^*, \x \rangle, \\
	\hat{F}_{k+1}(\x) &= \hat{f}_{k+1}(\x) - \langle \Am^\top \w_{k+1}^*, \x \rangle = \sum_{i = 1}^I \ell_i F_{k + 1 - i}(\x).
\end{align*}
Using the sub-multiplicativity of the norm we have $\norm{\psi(\w_{k+1}^*)} \leq \norm{\Am} \norm{\bar{\x} - \bar{\bar{\x}}}$, and we need to bound $\norm{\bar{\x} - \bar{\bar{\x}}}$.

Defining $\Gamma(\y) := \nabla_\x F_{k+1}(\y)$ and $\gamma(\y) := \nabla_\x \hat{F}_{k+1}(\y) - \nabla_\x F_{k+1}(\y)$, we can see that $\bar{\x}$ and $\bar{\bar{\x}}$ are the solutions of the generalized equation $(\Gamma + \gamma)(\y) = \p$ when $\p = \pmb{0}$ and $\p = \gamma(\bar{\bar{\x}})$. Therefore, applying the inverse function theorem \cite[Theorem~1A.1]{doncev_implicit_2014} we have $\norm{\bar{\x} - \bar{\bar{\x}}} \leq \norm{\gamma(\bar{\bar{\x}})} / \mu$. Finally, we have
$$
	\norm{\gamma(\bar{\bar{\x}})} = \norm{\nabla_\x \hat{F}_{k+1}(\bar{\x}) - \nabla_\x F_{k+1}(\bar{\x})} = \norm{\nabla_\x \hat{f}_{k+1}(\bar{\x}) - \nabla_\x f_{k+1}(\bar{\x})} \leq C(I) \Ts^I
$$
where the inequality holds by~\cref{eq:extrapolation-bound}.

Putting everything together yields the prediction error bound
$$
	\norm{\hat{\w}_{k+1}^* - \w_{k+1}^*} \leq \norm{\Am} C(I) \Ts^I / (\mu \bar{\mu}) =: \bar{\tau}
$$
and substituting into~\cref{eq:dual-general-tracking-error} yields the dual convergence bound.

The primal convergence bound can then be derived as a consequence of~\cref{eq:dual-general-tracking-error} by using the fact that \cite[Lemma~A.1]{bastianello_primal_2020}
\begin{align*}
	\norm{\x_{k+1} - \x_{k+1}^*} &\leq {\norm{\Am}}/{\mu} \norm{\w_{k+1} - \w_{k+1}^*}, \\
	\norm{\Bm (\y_{k+1} - \y_{k+1}^*)} &\leq \norm{\Bm}( {1}/{\rho} + {\norm{\Am}^2}/{\mu} ) \norm{\w_{k+1} - \w_{k+1}^*}\ \mathrm{\text{(in case $\gh \not\equiv 0$)}}.
\end{align*} \endstatement

\bibliographystyle{elsarticle-harv} 
\bibliography{references}

\end{document}